\documentclass{amsart}
\usepackage{amsmath, amssymb}
\usepackage[mathscr]{eucal}
\usepackage{color}
\usepackage{xcolor}
\usepackage{amsfonts,amscd,enumerate}
\usepackage{tikz, graphicx, pgfplots}
\pgfplotsset{compat=1.3}
\newtheorem{thm}{Theorem}

\newtheorem{ex}[thm]{Example}

\newtheorem{remark}[thm]{Remark}
\newtheorem{cor}[thm]{Corollary}

\begin{document}
\title[Expander Mixing Lemma for Directed Graphs]{An Expander Mixing Lemma for Directed Graphs via Transition Matrices}
\author[Rebecca Carter]{Rebecca Carter}
\address{School of Computing Sciences and Mathematics, Mount Royal University, Calgary, Alberta, T3E 6K6, Canada.}
\email{bcarter@mtroyal.ca }

\begin{abstract}  
We establish an Expander Mixing Lemma for directed graphs in terms of the eigenvalues of a transition matrix, extending the classical spectral inequality to the non-symmetric setting. This formulation works with the complex spectrum of the transition matrix and incorporates the conditioning of its eigenbasis through a condition number. For connected, non-bipartite, undirected regular graphs, the result reduces to the classical Expander Mixing Lemma. More generally, the theorem provides a framework for studying structural properties of directed graphs through the spectra of non-normal matrices. An earlier version of this paper contained an application to the toughness of directed graphs; that application has been withdrawn because its proof contained errors. 
\end{abstract}

\maketitle

\section{Introduction}

Spectral graph theory has traditionally focused on undirected graphs, for which many of the classical associated matrices are symmetric. In this setting, orthogonal eigenbases and real eigenvalues provide powerful tools for relating combinatorial structure to spectral properties. The spectral study of directed graphs and non-reversible Markov chains remains substantially less developed. The primary difficulty is that the matrices naturally associated to directed graphs are generally non-normal and therefore need not admit orthonormal eigenbases. 

Nevertheless, directionality is a fundamental feature of many real-world networks, making the extension of spectral methods to non-normal matrices both natural and necessary. Existing approaches to studying directed graphs spectrally are usually in terms of the eigenvalues of associated Hermitian matrices, such as the directed Laplacian \cite{Chung2005}, or by studying singular values \cite{Butler}. 

In this paper, we study directed graphs through the eigenvalues and eigenvectors of an associated transition matrix. The transition matrix of a directed graph does not admit an orthonormal basis of eigenvectors in general. Throughout, the transition matrix is viewed simultaneously as describing the stochastic process of a random walk and as a weighted adjacency matrix for the directed graph. Our goal is to extend classical spectral inequalities to this asymmetric setting while working directly with the eigenvalues of a non-normal matrix. 

One of the classical results of spectral graph theory is the Expander Mixing Lemma (EML), often attributed to Alon and Chung \cite{AlonChung1988}. The classical EML provides a quantitative relationship between the edge distribution of a $k$-regular graph and the spectral gap of its adjacency matrix. Numerous variations and extensions of this lemma have since been developed. In particular, the recent framework of Abiad and Zeijlemaker \cite{AbiadZeijlemaker2024} unifies several forms of the EML for irregular graphs, while Butler \cite{Butler} establishes a directed analogue formulated in terms of singular values. 

The purpose of this paper is to establish an Expander Mixing Lemma for directed graphs formulated directly in terms of the eigenvalues of their transition matrices. Since a directed graph's transition matrix need not possess an orthogonal eigenbasis, the resulting inequality necessarily depends not only on the spectrum itself but also on the geometry of the eigenbasis. This dependence appears through the condition number of a matrix of eigenvectors, which increases as the eigenvectors depart from orthogonality. 

More precisely, in Theorem \ref{EML_thm} we prove an Expander Mixing Lemma for irreducible, aperiodic, diagonalizable transition matrices. For a connected, non-bipartite, undirected $k$-regular graph, the transition matrix is symmetric, the eigenbasis can be chosen so that the condition number equals one, and our result recovers the classical Expander Mixing Lemma. Thus, the theorem may be viewed as an extension of the classical spectral inequality from symmetric adjacency matrices to asymmetric transition matrices. 

More broadly, these results suggest that many techniques from spectral graph theory may admit natural extensions to directed graphs without using Hermitian matrices. The framework developed here provides a basis for further investigations into how the spectra of non-Hermitian transition matrices encode structural properties of directed graphs. 

An earlier version of this paper included a proposed application to the toughness of directed graphs. That argument contained errors, and the application has therefore been withdrawn from the present version.

\section{Preliminaries}

Suppose $\Gamma$ is a directed graph with vertex set $V(\Gamma)$ and edge set $E(\Gamma)$. A walk in $\Gamma$ is an alternating sequence of vertices and edges
$$ v_0, e_1, v_1, \ldots, e_\ell, v_\ell,$$
where for each $1 \leq i \leq \ell$, the tail and head of the edge $e_i$ are $v_{i-1}$ and $v_i$, respectively. A random walk on the directed graph $\Gamma$ is described by a row stochastic transition matrix $P$, which can also be interpreted as a weighted adjacency matrix. Although such a matrix is not necessarily unique, Example~\ref{randomwalk_ex} presents a common construction.

\begin{ex} \label{randomwalk_ex}
    Suppose $\Gamma$ is a directed graph with vertex set $V(\Gamma)$ and edge set $E(\Gamma)$, in which every vertex has at least one outgoing edge. The matrix $P$ is a transition matrix for $\Gamma$ with entries  
    \begin{equation*} 
    p_{ij} = 
    \begin{cases}
        \frac{1}{d^-(v_i)}, & \text{if } (v_i, v_j) \in E(\Gamma), \\[6pt]
        0, & \text{otherwise},
    \end{cases}
    \end{equation*}
    where $d^-(v_i)$ is the outdegree of the vertex $v_i \in V(\Gamma)$. 
\end{ex}

Unless otherwise stated, we will take $P$ to be of this form and refer to it as the \textit{transition matrix} of $\Gamma$. 

The Perron--Frobenius theorem states that if $A$ is a nonnegative, irreducible matrix, then $A$ has a real positive eigenvalue $\lambda_1$ such that the absolute values of all other eigenvalues are less than or equal to $\lambda_1$. We will refer to this eigenvalue as the \textit{dominant eigenvalue}. An eigenvector corresponding to $\lambda_1$ can be chosen with all positive entries; it is often called the \emph{Perron vector}. 

For the transition matrix $P$ of a strongly connected directed graph $\Gamma$, the matrix $P$ is irreducible and the dominant eigenvalue is $\lambda_1 = 1$, with right eigenvector $\mathbf{1}$. The corresponding left Perron vector can be rescaled so that its components sum to 1, yielding the unique \textit{stationary distribution} $\boldsymbol{\pi}^T$ of $P$, which satisfies $\boldsymbol{\pi}^T P = \boldsymbol{\pi}^T$. If $P$ is in addition aperiodic, then all other eigenvalues of $P$ have absolute value strictly less than 1, and by the Fundamental Theorem of Markov Chains,
$$\pi_j = \lim_{t \to \infty} p_{ij}^{(t)},$$
where $p_{ij}^{(t)}$ denotes the $(i,j)$-entry of $P^t$ and $\pi_j$ is the $j$th component of $\boldsymbol{\pi}^T$. For any subset $U \subseteq V(\Gamma)$, we write $\pi(U) := \sum_{v_i \in U} \pi_i$.

An undirected graph can be viewed as a special case of a directed graph in which each undirected edge $\{v_i, v_j\}$ is replaced by two directed edges $(v_i, v_j)$ and $(v_j, v_i)$.

\section{Expander Mixing Lemma}

Let $P$ be a transition matrix for a directed graph $\Gamma$ with $n$ vertices, and let $\lambda_1 = 1, \lambda_2, \ldots, \lambda_n$ be the eigenvalues of $P$, listed with multiplicity, where $\lambda_1$ is the dominant eigenvalue. We define
$$\rho := \max_{2 \leq i \leq n} \lvert \lambda_i \rvert.$$
When $P$ is irreducible, the dominant eigenvalue $\lambda_1$ is simple, and when $P$ is in addition aperiodic, $\rho < 1$.

We will use $\lVert \cdot \rVert$ to denote the operator norm induced by the Euclidean norm. That is, for a matrix $A$, 
$$\lVert A \rVert = \sup_{\lVert x \rVert = 1} \lVert Ax \rVert_2.$$

\begin{thm} \label{EML_thm}
    Let $\Gamma$ be a directed graph with $n$ vertices and $P$ be its transition matrix with $(i,j)$-entries $p_{ij}$. Suppose $P$ is irreducible, aperiodic, and has $n$ linearly independent eigenvectors. Then, for any $U, W \subseteq V(\Gamma)$, 
    $$ \left\lvert \sum_{\substack{i \in U \\ j \in W}} p_{ij} - \lvert U \rvert \pi(W) \right\rvert \leq \rho \sqrt{\bigg(\lVert C \rVert^2 \lvert U \rvert - \frac{\lvert U \rvert ^2}{n}\bigg)\bigg(\lVert C^{-1} \rVert^2 \lvert W \rvert - \pi(W)^2 n \bigg)}, $$
    where $C$ is a matrix of normalized eigenvectors which diagonalizes $P$.
\end{thm}

The inequality can be further bounded to obtain a simpler form.

\begin{cor}\label{simpleEML_cor}
    Under the hypotheses of Theorem~\ref{EML_thm}, for any $U, W \subseteq V(\Gamma)$,
    $$\left\lvert \sum_{\substack{i \in U \\ j \in W}} p_{ij} - \lvert U \rvert \pi(W) \right\rvert \leq \rho \sqrt{ \lvert U \rvert \lvert W \rvert}\,\kappa(C),$$
where $\kappa(C) = \lVert C \rVert \lVert C^{-1} \rVert$ is the \textit{condition number} of $C$.
\end{cor} 

\begin{remark}
    The formulation above incorporates the conditioning of the eigenbasis. When $P$ is symmetric, eigenvectors corresponding to distinct eigenvalues are orthogonal; hence, if the columns of $C$ corresponding to each eigenspace are chosen to be orthonormal (see the discussion following Corollary~\ref{AlonChung_cor}), then $C$ is unitary and $\kappa(C)=1$. In the asymmetric case, $\kappa(C)$ captures how far the eigenbasis deviates from orthogonality.
\end{remark}

\begin{proof}[Proof of Theorem \ref{EML_thm}]
    We will start by constructing the matrix $C$. Let $x_1, \ldots, x_n$ be a set of normalized, linearly independent eigenvectors of $P$ where $x_i$ is an eigenvector for $\lambda_i$. According to the Perron--Frobenius Theorem, $\lambda_1=1$, and since $P$ is row stochastic and $\lambda_1$ is simple, we can set $x_1 = \frac{1}{\sqrt{n}}\mathbf{1}$. Let $C$ be the matrix whose columns are $x_1, \ldots, x_n$. Then $C$ is non-singular and $P=C\text{diag}(\lambda_1, \ldots, \lambda_n)C^{-1}$. Let $y_i^T$ be the $i$-th row vector of $C^{-1}$. By the Fundamental Theorem of Markov Chains, the $j$-th component of $P$'s unique stationary distribution $\boldsymbol{\pi}^T$ is $\pi_j = \lim_{t \to \infty} p_{ij}^{(t)}$ where $p_{ij}^{(t)}$ is the $(i, j)$-entry of $P^t$. Since $P$ is irreducible and aperiodic, according to the Perron--Frobenius Theorem, $\lvert \lambda_i \rvert < 1$ for $2 \leq i \leq n$. Therefore, 
    $$\lim_{t \to \infty} P^t = \lim_{t \to \infty} C \text{diag}(\lambda_1^t, \ldots, \lambda_n^t) C^{-1} = C \text{diag}(1, 0, \ldots, 0) C^{-1}.$$ 
    Hence, 
    \begin{align*}
        \pi_j &= \lim_{t \to \infty}p_{ij}^{(t)} \\
        &= (x_1)_i(y_1^T)_j \\
        &= \frac{(y_1^T)_j}{\sqrt{n}}
    \end{align*}
    and so $y_1^T = \sqrt{n}\boldsymbol{\pi}^T$. 
    
    Next, we will derive an expression for $$\sum_{\substack{i \in U \\ j \in W}} p_{ij}.$$ For a subset $S$ of $V(\Gamma)$ we define the characteristic vector $\chi_S$ component wise by,
    $$ (\chi_S)_i = 
        \begin{cases}
            1, & \text{if } v_i \in S \\[4pt]
            0, & \text{otherwise.}
        \end{cases}$$
    Then, for the sets $U, W \subseteq V(\Gamma)$ we have, 
    \begin{align*}
        \sum_{\substack{i \in U \\ j \in W}} p_{ij} &= \chi_U^T P \chi_W \\ 
        &= \sum_{i=1}^n  \chi_U^T (\lambda_i x_i y_i^T) \chi_W \\
        &= \chi_U^T \mathbf{1}\boldsymbol{\pi}^T \chi_W + \sum_{i=2}^n \lambda_i\, \chi_U^T x_i y_i^T \chi_W \\
        &= \lvert U \rvert \pi(W) + \sum_{i=2}^n \lambda_i\, \chi_U^T x_i y_i^T \chi_W.
    \end{align*}
    Taking absolute values and then applying the Cauchy--Schwarz inequality yields  
    \begin{align*}
        \left\lvert \sum_{\substack{i \in U \\ j \in W}} p_{ij} - \lvert U \rvert \pi(W) \right\rvert &\leq  \sum_{i=2}^n \vert \lambda_i \rvert\,  \lvert \chi_U^T x_i \rvert \lvert y_i^T \chi_W \rvert \\
        &\leq \rho \sqrt{\sum_{i=2}^n \lvert \chi_U^T x_i \rvert ^2 \sum_{i=2}^n \lvert y_i^T \chi_W \rvert ^2}.
    \end{align*}
    Note that $\chi_U^T x_i$ is the $i^{\text{th}}$ component of the row vector $\chi_U^T C$. Hence, 
    \begin{align*}
        \sum_{i=1}^n \lvert \chi_U^T x_i \rvert^2 &= \lVert \chi_U^T C \rVert ^2 \\
        &= \lVert C^T \chi_U \rVert ^2 \\
        &\leq \lVert C \rVert ^2 \lVert \chi_U \rVert ^2 \\
        &= \lvert U \rvert \lVert C \rVert ^2,
    \end{align*}
    where we used $\lVert C^T \rVert = \lVert C \rVert$. Recalling the definition of $x_1$ we have that, 
    $$\lvert \chi_U^T x_1 \rvert^2 = \frac{\lvert U \rvert ^2}{n}.$$ 
    Therefore, 
    $$\sum_{i=2}^n \lvert \chi_U^T x_i \rvert^2 \leq \lvert U \rvert \lVert C \rVert ^2 - \frac{\lvert U \rvert^2}{n}.$$
    Similarly, $y_i^T \chi_W$ is the $i^{\text{th}}$ component of $C^{-1}\chi_W$, so
    $$\sum_{i=1}^n \lvert y_i^T \chi_W \rvert^2 = \lVert C^{-1} \chi_W \rVert^2 \leq \lVert C^{-1} \rVert^2 \lvert W \rvert,$$
    and since $y_1^T = \sqrt{n}\boldsymbol{\pi}^T$, we have $\lvert y_1^T \chi_W \rvert^2 = \pi(W)^2 n$. Therefore,
    $$\sum_{i=2}^n \lvert y_i^T \chi_W \rvert^2 \leq \lvert W \rvert \lVert C^{-1} \rVert ^2 - \pi(W)^2n.$$
    We can conclude that  
    $$ \left\lvert \sum_{\substack{i \in U \\ j \in W}} p_{ij} - \lvert U \rvert \pi(W) \right\rvert \leq \rho \sqrt{\bigg(\lVert C \rVert^2 \lvert U \rvert - \frac{\lvert U \rvert ^2}{n}\bigg)\bigg(\lVert C^{-1} \rVert^2 \lvert W \rvert - \pi(W)^2 n \bigg)} .$$
\end{proof}

To prove Corollary \ref{simpleEML_cor}, we make a minor adjustment to the previous proof. 
\begin{proof}[Proof of Corollary \ref{simpleEML_cor}]
    From the proof of Theorem \ref{EML_thm} we see that $\lVert C \rVert^2 \lvert U \rvert \geq \frac{\lvert U \rvert^2}{n}$ and $\lVert C^{-1} \rVert^2 \lvert W \rvert \geq \pi(W)^2n$, since the differences are bounded below by sums of squares. Therefore, 
    \begin{align*}
        \left\lvert \sum_{\substack{i \in U \\ j \in W}} p_{ij} - \lvert U \rvert \pi(W) \right\rvert &\leq \rho \sqrt{\bigg(\lVert C \rVert^2 \lvert U \rvert - \frac{\lvert U \rvert ^2}{n}\bigg)\bigg(\lVert C^{-1} \rVert^2 \lvert W \rvert - \pi(W)^2 n \bigg)} \\ 
        &\leq \rho\sqrt{(\lVert C \rVert ^2 \lvert U \rvert)(\lVert C^{-1} \rVert^2 \lvert W \rvert)} \\
        &= \rho \sqrt{\lvert U \rvert \lvert W \rvert}\, \kappa(C).
    \end{align*}
    
\end{proof}

Consider now a connected, non-bipartite, undirected $k$-regular graph. Its transition matrix is $P=\frac{1}{k}A$, where $A$ is the adjacency matrix of the graph, and $P$ is irreducible and aperiodic. In this setting, $\rho = \mu/k$, where $\mu$ is the second-largest adjacency eigenvalue in absolute value. Moreover, since $P$ is symmetric, it admits an orthonormal eigenbasis containing $\frac{1}{\sqrt{n}}\mathbf{1}$, hence we can choose the columns of $C$ to be orthonormal so that $\lVert C \rVert = \lVert C^{-1} \rVert = 1$. The stationary distribution for $P$ is uniform, giving $\pi(W) = \frac{|W|}{n}$. Multiplying the inequality of Theorem~\ref{EML_thm} by $k$, we recover the classical Expander Mixing Lemma for such graphs. (The classical statement holds for all $k$-regular graphs.)

\begin{cor}[Expander Mixing Lemma; cf.\ Alon and Chung \cite{AlonChung1988}] \label{AlonChung_cor}
    Let $\Gamma$ be a connected, non-bipartite, undirected $k$-regular graph on $n$ vertices. Let $k=\theta_1 \geq \cdots \geq \theta_n$ be the adjacency eigenvalues. For any two subsets $U, W \subseteq V(\Gamma)$ let $e(U, W)$ denote the number of ordered pairs $(u, w)$ with $u \in U$, $w \in W$ and $u$ adjacent to $w$. Then, 
    $$\left\lvert e(U, W) - \frac{k \lvert U \rvert \lvert W \rvert}{n} \right\rvert \leq \mu \sqrt{\lvert U \rvert \lvert W \rvert \Big(1-\frac{\lvert U \rvert}{n}\Big) \Big(1 - \frac{\lvert W \rvert}{n}\Big)},$$
    where $\mu = \max_{2 \leq i \leq n}\lvert \theta_i \rvert$.  
\end{cor}

\subsection*{Dependence on the choice of eigenbasis}

There is, of course, some freedom in the choice of eigenbasis, and the bound in Theorem~\ref{EML_thm} depends on the norms $\lVert C \rVert$ and $\lVert C^{-1} \rVert$. A natural question, then, is how the bound changes with the choice of eigenbasis. 

When all eigenvalues are simple, any two normalized eigenbases can differ only by permuting the eigenvectors or multiplying each by a complex scalar of modulus one. These operations correspond to right multiplication by a unitary matrix and therefore do not alter $\lVert C \rVert$ or $\lVert C^{-1} \rVert$. Consequently, the bound remains unchanged. 

In the case of repeated eigenvalues, however, $\lVert C \rVert$ and $\lVert C^{-1} \rVert$ are dependent on the choice of eigenbasis. To make the bound invariant under different choices of eigenbasis within an eigenspace, we adopt the convention of choosing an orthonormal basis for each eigenspace. Any two such choices differ by right multiplication by a unitary matrix, and hence yield the same $\lVert C \rVert$ and $\lVert C^{-1} \rVert$. Under this convention, the bound in Theorem~\ref{EML_thm} is independent of the choice of eigenbasis.

\section{Concluding Remarks}

The results presented here demonstrate that the Expander Mixing Lemma extends naturally to directed graphs when formulated in terms of transition matrices. In contrast to approaches based on Hermitian matrices or singular values, the framework developed here works directly with the complex eigenvalues of the underlying matrix. 

A central feature of the resulting inequality is the appearance of the condition number of the eigenbasis. This highlights an important distinction between normal and non-normal spectral theory: for directed graphs, the geometry of the eigenbasis itself becomes an essential spectral parameter. 

The theorem established here recovers the classical Expander Mixing Lemma for connected, non-bipartite, undirected regular graphs and suggests several possible directions for future work. In particular, it would be interesting to investigate whether analogous spectral techniques can be used to obtain structural results for directed graphs. 

More broadly, these results suggest that many classical tools from spectral graph theory may admit natural extensions to directed graphs in terms of their non-Hermitian matrices. 

\vspace{0.6cm}
\noindent{\textit{Acknowledgments.}} I am grateful to Professor M. Ram Murty (Queen's University) and Nic Fellini (Queen's University) for their helpful comments and suggestions on earlier drafts of this work.

\end{document}